\documentclass[11pt]{amsart}

\usepackage{amsfonts,amssymb,amsmath,amsthm,epsfig,euscript}

\usepackage{latexsym}
\usepackage{eepic,epic}
\usepackage{url}

\newtheorem{theorem}{Theorem}
\newtheorem{lemma}[theorem]{Lemma}
\newtheorem{corollary}{Corollary}

\newtheorem{proposition}[theorem]{Proposition}
\newtheorem{observation}{Observation}

\theoremstyle{definition}
{
\newtheorem{definition}{Definition}

}

\long\def\symbolfootnote[#1]#2{\begingroup
\def\thefootnote{\fnsymbol{footnote}}\footnote[#1]{#2}\endgroup}

\newcommand{\nolncs}[1]{}

\title{Semi-Transitive Orientations and Word-Representable Graphs$^{*}$}
\thanks{$*$ A preliminary version of this
    work appeared in the International Workshop on Graph-Theoretic Concepts in Computer Science (WG), 2008.}

\urldef{\magnusmail}\url{mmh@ru.is}
\urldef{\sergeymail}\url{sergey.kitaev@cis.strath.ac.uk}
\urldef{\artemmail}\url{artem@math.nsc.ru}

\author[M. M. Halld{\'o}rsson]{Magn{\'u}s M. Halld{\'o}rsson$^{\dagger}$}
\thanks{$\dagger$ School of Computer Science, Reykjavik University, 101 Reykjavik, Iceland. Email: \magnusmail.
   Supported by grant-of-excellence 120032011 from the Iceland Research Fund.}
\author[S. Kitaev]{Sergey Kitaev$^{\ddag}$}
\thanks{$\ddag$ Department of Computer and Information Sciences, University of Strathclyde, Glasgow G1 1XH, UK. Email: \sergeymail}

\author[A. Pyatkin]{Artem Pyatkin$^{\S}$}
\thanks{$\S$ Sobolev Institute of Mathematics; Novosibirsk State University; Novosibirsk, Russia. Email: \artemmail}

\begin{document}


\maketitle


\begin{abstract}
A graph $G=(V,E)$ is a \emph{word-representable graph} if there exists a word $W$ over
the alphabet $V$ such that letters $x$ and $y$ alternate in $W$ if
and only if $(x,y)\in E$ for each $x\neq y$.

In this paper we give an effective characterization of word-representable
graphs in terms of orientations. Namely, we show that a graph is
word-representable if and only if it admits a \emph{semi-transitive
orientation} defined in the paper.
This allows us to prove a number of results about word-representable graphs,
in particular showing that the recognition problem is in NP,
and that word-representable graphs include all 3-colorable graphs.

We also explore bounds on the size of the word representing the graph.
The representation number of $G$ is the minimum
$k$ such that $G$ is a representable by a word, where each letter occurs $k$ times; such a $k$ exists for any word-representable graph.
We show that the representation
number of a word-representable graph on $n$ vertices is at most $2n$, while there exist graphs for which it
is $n/2$. \\


\end{abstract}


\section{Introduction}

A graph $G=(V,E)$ is \emph{word-representable} if there exists a word $W$
over the alphabet $V$ such that  for each pair of distinct letters $x$ and $y$,
$(x,y)\in E$ if and only if the occurrences of the letters alternate in $W$.
As an example, the words $abcdabcd$, $abcddcba$, and $abdacdbc$ represent the 4-clique, $K_4$; 4-independent set, $\overline{K_4}$; and the 4-cycle, $C_4$, labeled by $a,b,c,d$ in clockwise direction, respectively.

If each letter appears exactly $k$ times in the word,
the word is said to be \emph{$k$-uniform} and the graph is said to be \emph{$k$-word-representable}.
It is known that any word-representable graph is $k$-word-representable, for some $k$ \cite{KP}.

The class of word-representable graphs is rich, and properly contains several important graph classes to be discussed next. \\

\begin{itemize}
\item[] {\bf Circle graphs.}
  Circle graphs are those whose vertices can be represented as chords on a circle in such a way that two nodes in the
  graph are adjacent if and only if the corresponding chords overlap.
  Assigning a letter to each chord and listing the letters in the order they appear along the circle, one obtains
  a word where each letter appears
  twice and two nodes are adjacent if and only if the letter occurrences alternate \cite{Courcelle}.
  Therefore, circle graphs are the same as 2-word-representable graphs.\\

\item[] {\bf Comparability graphs.} A comparability graph is one that admits a transitive orientation of the edges, i.e. an assignment of directions to the edges such that the adjacency relation of the resulting digraph is transitive:
the existence of arcs $xy$ and $yz$ yields that $xz$ is an arc.
Such a digraph induces a poset on the set of vertices $V$. Note that each poset is an intersection of several linear orders and each linear order corresponds to some permutation $P_i$ of $V$.
These permutations can be concatenated to a word of the form $P_1P_2\cdots P_k$. Then two letters alternate in this word if and only if they are in the same order in each permutation (linear order), and this means that they are comparable in the poset and, thus, the corresponding letters are adjacent in the graph. So, comparability graphs are a subclass of word-representable graphs that is known as the class of \emph{permutationally representable graphs} in the literature.\\

\item[] {\bf Cover graphs.}
The ({\em Hasse}) \emph{diagram} of a partial order $P=(V,<)$ is the directed graph on $V$ with an arc from $x$ to $y$ if $x < y$ and there is no $z$ with $x < z < y$ (in which case $x$ ``covers'' $y$). A graph is a \emph{cover graph} if it can be oriented as a diagram of a partial order.
Limouzy \cite{L}
observed that cover graphs are exactly the triangle-free word-representable graphs.\\

\item[] {\bf 3-colorable graphs.} A corollary of our main structural result in this paper
is that the class of word-representable graphs contains all 3-colorable graphs.\\
\end{itemize}

Various computational hardness results follow from these inclusions. Most importantly, since
it is an NP-hard problem to recognize cover graphs \cite{Brightwell,NR}, the same holds for word-representable graphs.
Also, the NP-hardness of various optimization problems, such as Independent Set, Dominating Set, Graph
Coloring, and Clique Partition, follows from the case of 3-colorable graphs.\\


\noindent {\bf Our results.}
The main result of the paper is an alternative characterization of word-representable graphs in terms of orientations.

A directed graph (digraph) $G=(V,E)$ is  \emph{semi-transitive} if
it is acyclic and for any directed path $v_1v_2\cdots v_k$, either
$v_1v_k\not\in E$ or $v_iv_j\in E$ for all $1\le i<j\le k$. Clearly,
comparability graphs (i.e. those admitting transitive orientations) are semi-transitive.
The main result of this paper is Theorem~\ref{thm:rep-equals-semi-trans} saying that a graph is word-representable if
and only if it admits a semi-transitive orientation.

The proof of the main result shows that any word-representable graph on $n\ge 3$ vertices is $(2n-4)$-word-representable.  This
bound implies that the problem of recognizing word-representable graphs is contained in NP.  Previously, no polynomial
upper bound was known on the \emph{representation number}, the smallest value $k$ such that the given graph is
$k$-word-representable.  Our bound on the representation number is tight up to a constant factor, as we construct graphs
with representation number $n/2$. We also show that deciding if a word-representable graph is $k$-word-representable is
NP-complete for $3 \le k \le n/2$.

One corollary of the structural result is that all 3-colorable graphs are word-representable graphs.  The containment
properly subsumes all previously known classes of word-representable graphs: outerplanar graphs, prisms, and comparability
graphs.  On the other hand, there are non-3-colorable graphs are not word-representable. \\

\noindent {\bf A motivating application.}
Consider a scenario with $n$ recurring tasks with requirements on the
alternation of certain pairs of tasks. This captures typical
situations in periodic scheduling, where there are recurring
\emph{precedence} requirements.

When tasks occur only once, the pairwise requirements form
precedence constraints, which are modeled by partial orders.
When the orientation of the constraints is omitted,
the resulting pairwise constraints form comparability graphs.
The focus of this paper is to study the class of undirected graphs
induced by the alternation relationship of recurring tasks.

Consider, e.g. the following five tasks that may be involved in 
operation of a given machine: 1) Initialize controller, 2) Drain
excess fluid, 3) Obtain permission from supervisor, 4) Ignite motor,
5) Check oil level. Tasks 1 \& 2, 2 \& 3, 3 \&4, 4 \& 5, and 5 \& 1
are expected to alternate between all repetitions of the events.
This is shown in Fig.~\ref{Fig1}(b), where each pair of alternating tasks is connected by an edge. One possible task execution
sequence that obeys these recurrence constraints -- and no other --
is shown in Fig.~\ref{Fig1}(a).
Later in the paper, we will introduce an orientation of
such graphs that we call a {\em semi-transitive orientation}; such an orientation for our example is shown in Fig.~\ref{Fig1}(c).

\begin{figure}[ht]
\begin{center}
\includegraphics[scale=0.6]{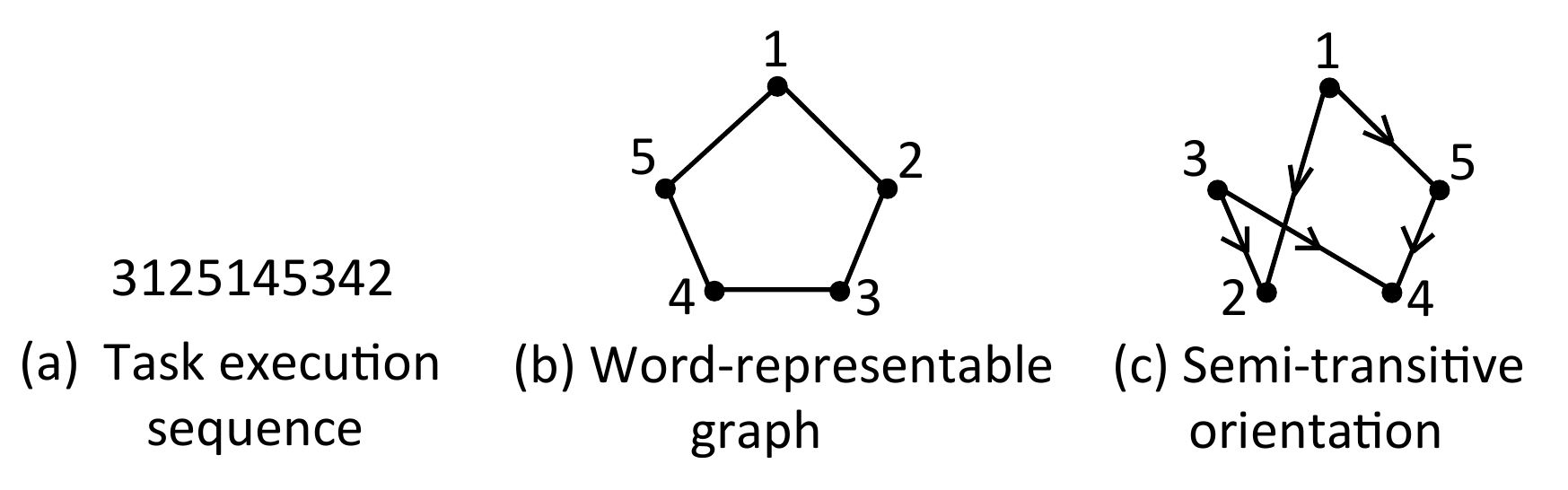}
\end{center}
\caption{The word in (a) corresponds to the word-representable graph in (b). A semi-transitive orientation of the graph is given in (c).}
\label{Fig1}
\end{figure}

Execution sequences of recurring tasks can be viewed as words over
an alphabet $V$, where $V$ is the set of tasks. \\


\noindent {\bf Related work.} The notion of directed word-representable graphs was introduced in~\cite{KS} to
obtain asymptotic bounds on the free spectrum of the widely-studied
{\em Perkins semigroup}, which has played central role in semigroup
theory since 1960, particularly as a source of examples and
counterexamples.
In \cite{KP}, numerous properties of word-representable graphs were derived
and several types of word-representable and non-word-representable graphs
pinpointed.
Some open questions from \cite{KP} were resolved recently in \cite{HKP}, including the representability of the Petersen graph.

Circle graphs were generalized to polygon-circle graphs (see \cite{Pergel}), which are the intersection graphs of polygons
inscribed in a circle. If we view each polygon as a letter and read the incidences of the polygons on the circle in
order, we see that two polygons intersect if and only if there \emph{exists} a pair of occurrences of the two polygons
that alternate. This differs from word-representable graphs where \emph{all} occurrences of the two letters must alternate in
order for the nodes to be adjacent.

Cyclic (or periodic) scheduling problems have been studied extensively
in the operations research literature
\cite{HanenMunier,Middendorf}, as well as in the AI
literature \cite{Draper}.
These are typically formulated with more general constraints, where,
e.g. the 10th occurrence of task A must be preceded by the 5th
occurrence of task B. The focus of this work is then on obtaining
effective periodic schedules, while maintaining a small cycle time.
We are, however, not aware of work on characterizing the graphs formed
by the cyclic precedence constraints.
A different periodic scheduling application related to word-representable graphs was
considered by Graham and Zang \cite{GZ}, whose work involves a counting
problem related to the cyclic movements of a robot arm. More
generally, given a set of jobs to be performed periodically, certain
pairs $(a,b)$ must be done alternately, e.g.\ since the product of
job $a$ is used as a resource for job $b$. Any valid execution
sequence corresponds to a word over the alphabet formed by the jobs.
The word-representable graph given by such a word must then contain the
constraint pairs as a subgraph.

The preliminary version of this work, that appeared in \cite{HKP2011}, claimed an upper bound of $n$ on the representation number.
This was based on a lemma (Lemma 2 in \cite{HKP2011}) that turned out to be false.
We present a corrected proof of the main result that gives a slightly weaker bound of $2n-4$ on the representation number. \\


\noindent {\bf Organization of the paper.} The paper is organized as follows. In Section~\ref{def}, we give
definitions of objects of interest and review some of the known
results.
In Section~\ref{char}, we give a characterization of word-representable
graphs in terms of orientations and discuss some important
corollaries of this fact. In Section~\ref{nice}, we examine the
representation number, and show that it is always at most $2n-4$, but can
be as much as $n/2$. We explore, in Section~\ref{classif}, which
classes of graphs are word-representable, and show, in particular,
that 3-colorable graphs are such graphs, but numerous other
properties are independent from the property of being word-representable. 
Finally, we conclude with 
two open problems in Section~\ref{open}.

\section{Definitions and Properties}
\label{def}

Let $W$ be a finite word.  If $W$ involves letters $x_1, x_2, \ldots, x_n$ then we write $A(W) = \{x_1, x_2,\ldots,
x_n\}$.  A word is $k$-\emph{uniform} if each letter appears in it exactly $k$ times. A 1-uniform word is also called a
{\it permutation}. Denote by $W_1W_2$ the concatenation of words $W_1$ and $W_2$. We say that  letters $x_i$ and
$x_j$ {\it alternate} in $W$ if the word induced by these two letters contains neither $x_ix_i$ nor $x_jx_j$ as a
factor. If a word $W$ contains $k$ copies of a letter $x$, then we denote these $k$ appearances of $x$ from left to right by $x^1,x^2,\ldots
,x^k$. We write $x_i^j<x_k^{\ell}$ if $x_i^j$ occurs in $W$ before $x_k^{\ell}$, i.e. $x_i^j$ is to the left of $x_k^{\ell}$ in $W$.

We say that a word $W$ \emph{represents} a graph $G=(V,E)$ if
there is a bijection $\phi:A(W)\rightarrow V$ such that
$(\phi(x_i),\phi(x_j))\in E$ if and only if $x_i$ and $x_j$
alternate in $W$.
We call a graph $G$ \emph{word-representable} if there
exists a word $W$ that represents $G$.
It is convenient to identify the vertices of a
word-representable graph and the corresponding letters of a word
representing it.
If $G$ can be represented by a
$k$-uniform word, then we say that $G$ is a \emph{$k$-word representable} graph and the word {\em $k$-represents} $G$.

The \emph{representation number} of a word-representable graph $G$
is the minimum $k$ such that $G$ is a $k$-word-representable graph. It follows from \cite{KP} that the representation number is well-defined for any word-representable graph.
We call a graph \emph{permutationally representable} if it can be
represented by a word of the form $P_1P_2\cdots P_k$, where each $P_i$
is a permutation over the same alphabet given by the graph vertices.

A digraph $D=(V,E)$ is {\em transitive} if the adjacency relation is
transitive, i.e.\ for any vertices $x,y,z\in V$, the existence of
the arcs $xy,yz\in E$ yields that $xz\in E$. A {\em comparability
graph} is an undirected graph that admits an orientation of the edges
that yields a transitive digraph.

The following properties of word-representable graphs and facts  from \cite{KP} are useful.
A graph $G$ is word-representable if and only if it is $k$-word representable for
some $k$.
If $W=AB$ is $k$-uniform word representing a graph $G$, then
the word $W'=BA$ also $k$-represents~$G$.

The wheel $W_5$ is the smallest
non-word-representable graph. Some examples of non-word-representable graphs on 6 and 7
vertices are given in Fig.~\ref{small}.

\begin{figure}[ht]
\begin{center}
\includegraphics[scale=0.6]{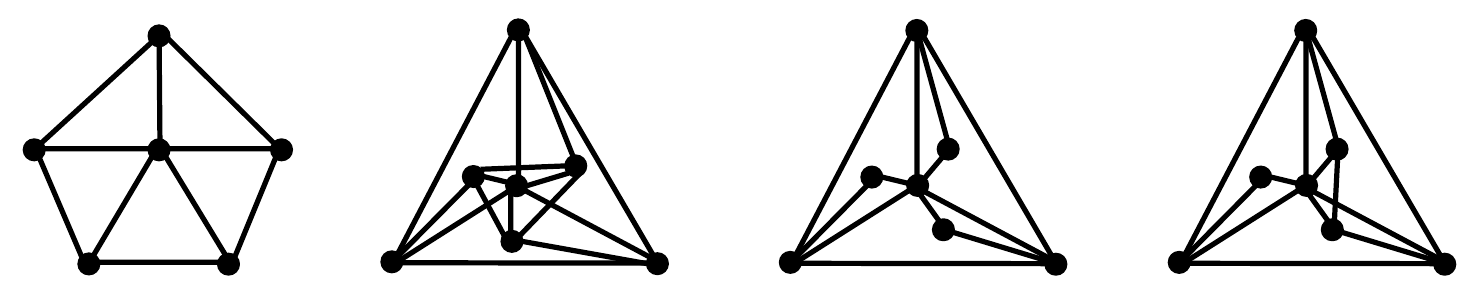}
\end{center}
\caption{Small non-word-representable graphs.}
\label{small}
\end{figure}


\section{Characterizing Word-Representable Graphs in Terms of Orientability}\label{char}

In this section we present a characterization,
which implies
that word-representability corresponds to a property of a digraph
obtained by directing the edges in a certain way.
It is known that a graph is permutationally representable 
if and only if it has a transitive orientation (i.e. is a comparability graph) \cite{KS}.
We prove a similar
fact on word-representable graphs, namely, that a graph is word-representable 
if and only if it has a certain {\em semi-transitive orientation} that we shall define.

Let $G=(V,E)$ be a graph.
An acyclic orientation of $G$ is \emph{semi-transitive} if
 for any directed path $v_1\rightarrow v_2\rightarrow \cdots \rightarrow v_k$ 
 either
\begin{itemize}
\item there is no arc $v_1\rightarrow v_k$, or
\item the arc $v_1\rightarrow v_k$ is present and there are arcs $v_i\rightarrow v_j$ for all $1\le i<j\le k$. That is,  in this case, the (acyclic) subgraph induced by the vertices $v_1,\ldots,v_k$ is a transitive clique (with the unique source $v_1$ and the unique sink $v_k$).
\end{itemize}
We call such an orientation a {\em semi-transitive orientation}. For example, the orientation of the graph in Figure \ref{fig:semi-transitive-orientation} is semi-transitive.
\begin{figure}[ht]
\begin{center}
\includegraphics[scale=0.6]{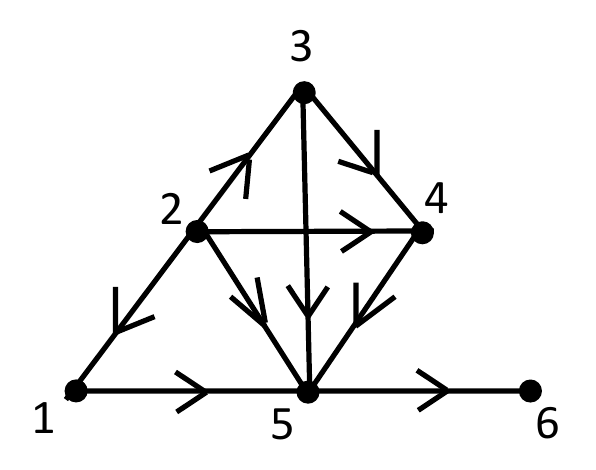}
\end{center}
\vspace{-20pt}
\caption{An example of a semi-transitive orientation.}
\label{fig:semi-transitive-orientation}
\end{figure}

A graph $G=(V,E)$ is \emph{semi-transitive} if it admits a semi-transitive orientation.

\begin{figure}[ht]
\begin{center}
\includegraphics[scale=0.6]{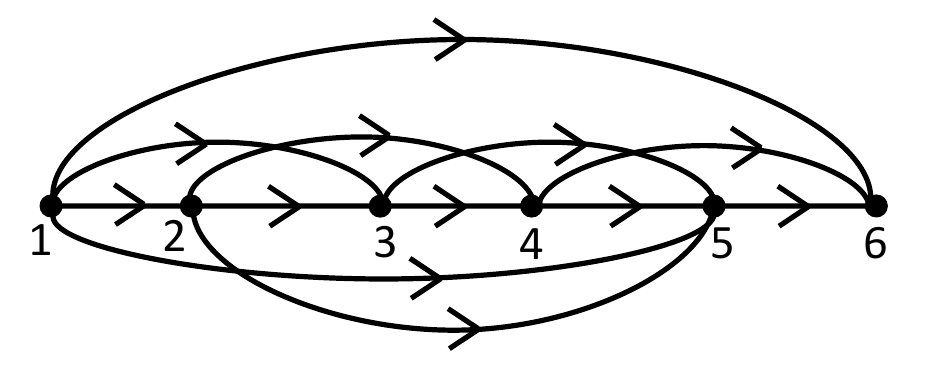}
\end{center}
\vspace{-20pt}
\caption{An example of a shortcut.}
\label{fig:shortcut}
\end{figure}

We can alternatively define semi-transitive orientations in terms of induced subgraphs.  A \emph{semi-cycle} is the directed
acyclic graph obtained by reversing the direction of one arc of a directed cycle. An acyclic digraph is a
\emph{shortcut} if it is induced by the vertices of a semi-cycle and contains a pair of non-adjacent vertices. Thus, a
digraph on the vertex set $\{ v_1, \ldots, v_k\}$ is a shortcut if it contains a directed path $v_1\rightarrow
v_2\rightarrow \cdots \rightarrow v_k$, the arc $v_1\rightarrow v_k$, and it is missing an arc $v_i\rightarrow v_j$
for some $1 \le i < j \le k$; in particular, we must have $k\geq 4$, so that any shortcut is on at least four
vertices. See Figure~\ref{fig:shortcut} for an example of a shortcut (there, the arcs $1\rightarrow 4$, $2\rightarrow
6$, and $3\rightarrow 6$ are missing). 

\begin{definition}
An orientation of a graph is semi-transitive, if it is acyclic and contains no shortcuts.
\end{definition}



For a word $W$, denote by $P(W)$ its {\it initial permutation}, i.e. the permutation obtained by removal from $W$ all but the first appearances of each letter. Let $D=(V,E)$ be an acyclic 
orientation of a graph $G$.
For vertices $u$ and $v$, let $u \leadsto v$ denote that there exists a directed path from $u$ to $v$ in $D$.
By definition, $u\leadsto u$.
We say that a permutation $P$ of the set $V$ is a {\it topological sort} of $D$ if for every distinct $u, v\in V$ such that $u \leadsto v$, the letter $u$ precedes $v$ in $P$.

We say that a word $W$ over the alphabet $V$ representing a graph $H$ \emph{covers} a set $A$ of non-edges of the digraph $D$ if the following four conditions hold:

1) $W$ is $k$-uniform for some $k$;

2) $P(W)$ is a topological sort of $D$;

3) $G$ is a subgraph of $H$;

4) each non-edge in $A$ is also a non-edge in $H$.

We need the following easy to prove lemma.

\begin{lemma}\label{covering}
  Let words $W_1$ and $W_2$ cover, respectively, sets $A$ and $B$ of non-edges of an acyclic digraph
  $D=(V,E)$. Then the word $W=W_1W_2$ covers the set of non-edges $A\cup B$. 
\end{lemma}

\begin{proof}
Let $W_1$ be $k$-uniform and $W_2$ be $\ell$-uniform. Then, clearly, $W$ is $(k+l)$-uniform. $P(W)$ is a topological sort of $D$ since $P(W)=P(W_1)$. If $uv\in A$ then $u$ and $v$ do not alternate in $W_1$. If $uv\in B$ then $u$ and $v$ do not alternate in $W_2$. So, for every $uv\in A\cup B$ the letters $u$ and $v$ cannot alternate in $W$. Finally, let $uv\in E(G)$. Then without loss of generality, it corresponds to an arc $u \rightarrow v$ in the digraph $D$. Then, by conditions 1)--3), the subwords of $W_1$ and $W_2$ induced by the letters $u$ and $v$ are alternating, starting with $u$ and ending with $v$. But then the same is true for the subword of $W$ induced by these letters, i.e. $u$ and $v$ alternate in $W$ and so, $uv\in E(H)$.
\end{proof}

Now we can prove our main technical lemma.

\begin{lemma}\label{lemma-nice}
  Let $D=(V,E)$ be a semi-transitively oriented graph and $v\in V$.  Then the non-edges incident with $v$ can be covered
  by a $2$-uniform word. 
\end{lemma}

\begin{proof}
Let $I(v)=\{u:u\rightarrow v\}$ be the set of all in-neighbors of $v$, and $O(v)=\{u:v\rightarrow u\}$ be the set of all out-neighbors of $v$.
Also, let $A(v) = \{ u \in V : u \leadsto v\} \setminus I$
be the set of $v$'s non-neighboring vertices that can reach $v$, and $B(v) = \{ u \in V : v \leadsto u\} \setminus O$ be the set of
$v$'s non-neighboring vertices that can be reached from $v$.
Finally, let $T(v) = V \setminus (\{v\}\cup I(v) \cup O(v) \cup A(v) \cup B(v))$ be
the set of remaining vertices. 
Note that the sets $I(v)$, $O(v)$, $A(v)$, $B(v)$ and $T(v)$ are pairwise disjoint and some of them can be empty.

Denote by $A, B, I, O$ and $T$
topological sorts of the corresponding digraphs induced by the sets $A(v)$, $B(v)$, $I(v)$, $O(v)$ and $T(v)$, respectively.

We now consider the 2-uniform word $W$ given by
\[   W = A\  I\  T\  A\  v\  O\  I\  v\  B\  T\  O\  B.\]

We claim that $W$ covers all non-edges incident with $v$, i.e. the non-edges of type $vu$ for $u\in T(v)\cup A(v)\cup B(v)$.

Condition 1) holds automatically. Clearly, $W$ represents the graph $H$ that is
the union of the cliques $T\cup A\cup I,\ T\cup B\cup O$, and $I\cup O\cup \{v\}$. Since $v$ is not adjacent to each $u\in T\cup A\cup B$ in $H$, condition 4) is true.
Let us check condition 3). Indeed there are no edges connecting $v$ with $T(v)\cup A(v)\cup B(v)$ in $G$ by the definition. Note also that no arcs can go from $A(v)$ to $O(v) \cup B(v)$, or from $I(v)$ to $B(v)$ in $D$,
since that would induce a shortcut. Also, there are no arcs in the other direction,
since the digraph $D$ is acyclic. So, $G$ has no edges connecting $A(v)$ with $O(v) \cup B(v)$, or $I(v)$ with $B(v)$. Since all other edges exist in $H$, we have that $G$ is a subgraph of $H$, i.e. condition 3) holds. Finally, let us check condition 2). We have $P(W)=AITvOB$. Let $u$ and $w$ be distinct vertices such that $u \leadsto w$. If $u$ and $w$ are in the same set $A,I,T,O,$ or $B$ then $u$ precedes $w$ in $P(W)$ since the corresponding set is a topological sort. If neither $u$ nor $w$ is in $T$ then $u$ precedes $w$ in $P(W)$ because otherwise the digraph $D$ would contain a directed cycle. If $u\in T$ then $w$ cannot be in $A\cup I$ since otherwise there would be a directed path from $u$ to $v$ and thus $u$ must be also in $A\cup I$ by the definition. So, $u$ precedes $w$ in $P(W)$. Finally, if $w\in T$ then $u$ cannot be in $\{v\}\cup O\cup B$ since otherwise there would be a directed path from $v$ to $u$ and from $u$ to $w$ and thus $w$ must be in $O\cup B$ by the definition. So, $u\in A\cup I$ and hence $u$ precedes $w$ in $P(W)$. Therefore, $P(W)$ is a topological sort of $D$ and condition 2) is true.
\end{proof}



Now we are ready to prove the main result.

\begin{theorem}
A graph $G$ is word-representable if and only if it is semi-transitively orientable. Moreover, each non-complete word-representable graph is $2(n-\kappa)$-word-representable where $\kappa$ is the size of the maximum clique in~$G$.
\label{thm:rep-equals-semi-trans}
\end{theorem}

\begin{proof}
For the forward direction, given a word $W$, we direct an edge of $G$
from $x$ to $y$ if the first occurrence of $x$ is before that of $y$
in the word. Let us show that such an orientation $D$ of $G_W$ is
semi-transitive. Indeed, assume that $x_0x_t\in E(D)$ and there is a
directed path $x_0x_1\cdots x_t$ in $D$. Then in the word $W$ we
have $x_0^i<x_1^i<\cdots<x_t^i$ for every $i$. Since $x_0x_t\in
E(D)$ we have $x_t^i<x_0^{i+1}$. But then for every $j<k$ and $i$
there must be $x_j^i<x_k^i<x_j^{i+1}$, i.e. $x_ix_j\in E(D)$. So,
$D$ is semi-transitive.

The other direction follows directly from Lemmas~\ref{covering} and~\ref{lemma-nice}. Indeed, let $K$ be a maximum clique of $G$. Denote by $D$ a
semi-transitive orientation of the graph $G$.
Let $W_v$ be the 2-uniform word that covers all the non-edges in $G$
incident with a vertex $v \in V\setminus K$.
Concatenating $n-\kappa$ such words $W_v$ induces a word $W$ that covers all
non-edges in $G$ and preserves all edges. This follows from the fact that every non-edge has at least one endpoint outside $K$.  Thus, $G$ is represented by $W$. Moreover, this word is $2(n-\kappa)$-uniform, proving the second part of the theorem.
\end{proof}

Since each complete graph is 1-word-representable and each edgeless graph (having maximum clicks of size 1) is 2-word-representanle, we have the following statement.

\begin{corollary}\label{thm:di-to-string}
Each word-representable graph $G$ on $n\ge 3$ vertices is $2(n-2)$-word-representable.
\end{corollary}




\section{The Representation Number of Graphs}\label{nice}

We focus now on the following question: Given a word-representable graph,
how large is its representation number? In~\cite{KP}, certain
classes of graphs were proved to be 2- or 3-word representable, and an
example was given of a graph (the triangular prism) with the
representation number of 3. More on graphs with representation number 3 can be found in~\cite{K}. On the other hand, no examples were
known of graphs with representation numbers larger than 3, nor were
there any non-trivial upper bounds known.


Theorem~\ref{thm:rep-equals-semi-trans} implies that the graph
property of word-representability is polynomially verifiable, i.e. 
the recognition problem is in NP.
Limouzy \cite{L} observed that triangle-free representable graphs are precisely the
cover graphs, i.e. graphs that can be oriented as the diagrams of a partial order.
Determining whether a graph is a cover graph is NP-hard \cite{Brightwell,NR}, 
and thus it is also hard to determine if a given (triangle-free) graph is word-representable.
Thus, we obtain the following exact classification.

\begin{corollary}
The recognition problem for word-representable graphs is NP-complete.
\label{cor:inNP}
\end{corollary}


We now show that there are graphs with representation number of $n/2$,
matching the upper bound within a factor of 4.

The \emph{crown graph} $H_{k,k}$ is the graph obtained from
the complete bipartite graph $K_{k,k}$ by removing a perfect
matching. Denote by $G_k$ the graph obtained from a crown
graph $H_{k,k}$ by adding a universal vertex (adjacent to all vertices in $H_{k,k}$).

\begin{theorem} \label{thm:example}
The graph $G_k$ has representation number $k=\lfloor n/2\rfloor$.
\end{theorem}

The proof is based on three statements.

\begin{lemma}
  Let $H$ be a graph and $G$ be the graph obtained from $H$ by adding an all-adjacent vertex.  Then $G$ is a
  $k$-word-representable graph if and only if $H$ is a permutational $k$-word-representable graph.
\label{obs:rep-permrep}
\end{lemma}

\begin{proof}
Let 0 be the letter corresponding to the all-adjacent vertex. Then
every other letter of the word $W$ representing $G$ must appear
exactly once between two consecutive zeroes. We may assume also that
$W$ starts with 0. Then the word $W\setminus \{0\}$, formed by deleting
all occurrences of 0 from $W$, is a permutational $k$-representation of $H$.
Conversely, if $W'$ is a
word permutationally $k$-representing $H$, then we insert 0 in front
of each permutation to get a (permutational) $k$-representation of $G$.
\end{proof}

Recall that the {\em order dimension} of a poset is the minimum
number of linear orders such that their intersection induces this
poset.

\begin{lemma}
  A comparability graph is a permutational $k$-word-representable graph if and only if the poset induced by this graph has
  dimension at most $k$.
\label{obs:dimension}
\end{lemma}

\begin{proof}
Let $H$ be a comparability graph and $W$ be a word permutationally
$k$-representing it.  Each permutation in $W$ can be considered as a
linear order where $a<b$ if and only if $a$ meets before $b$ in the permutation.
We want to show that the comparability graph of the poset
induced by the intersection of these linear orders coincides with $H$.

Two vertices $a$ and $b$ are adjacent in $H$ if and only if their
letters alternate in the word. So, they must be in the same order in
each permutation, i.e. either $a<b$ in every linear order or $b<a$
in every linear order. But this means that $a$ and $b$ are comparable in the
poset induced by the intersection of the linear orders, i.e. $a$ and
$b$ are adjacent in its comparability graph.
\end{proof}

The next statement most probably is known but we give its proof here for
the sake of completeness.

\begin{lemma} \label{cocktail}
  Let $P$ be the poset over the $2k$ elements $\{a_1,a_2,\ldots,a_k,b_1,b_2,\ldots,b_k\}$ such that $a_i<b_j$ for every $i\ne j$
  and all other elements are not comparable.
  Then, $P$ has dimension $k$.
\end{lemma}

\begin{proof}
Assume that this poset is the intersection of $t$ linear orders.
Since $a_i$ and $b_i$ are not comparable for each $i$, their must be
a linear order where $b_i<a_i$. If we have in some linear order both
$b_i<a_i$ and $b_j<a_j$ for $i\ne j$, then either $a_i<a_j$
or $a_j<a_i$ in it. In the first case we have that $b_i<a_j$, in the
second that $b_j<a_i$. But each of these inequalities contradicts
the definition of the poset. Therefore, $t\ge k$.

In order to show that $t=k$ we can consider a linear order
$a_1<a_2<\ldots<a_{k-1}<b_k<a_k<b_{k-1}<\ldots<b_2<b_1$ together
with all linear orders obtained from this order by the simultaneous
exchange of $a_k$ and $b_k$ with $a_m$ and $b_m$ respectively
($m=1,2,\ldots,k-1$). It can be verified that the intersection of
these $k$ linear orders coincides with our poset.
\end{proof}

Now we can prove Theorem~\ref{thm:example}. Since the crown graph
$H_{k,k}$ is a comparability graph of the poset $P$, we deduce from
Lemmas~\ref{cocktail} and~\ref{obs:dimension} that $H_{k,k}$ is
a permutational $k$-word-representable graph but not a permutational
$(k-1)$-word-representable graph. Then by Lemma~\ref{obs:rep-permrep} we have that
$G_k$ is a $k$-word-representable graph but not a $(k-1)$-word-representable graph.
Theorem~\ref{thm:example} is proved. \qed
\medskip

The above arguments help us also in deciding the complexity of
determining the representation number. From Lemmas
\ref{obs:rep-permrep} and \ref{obs:dimension}, we see that it is as
hard as determining the dimension $k$ of a poset. Yannakakis
\cite{Yann} showed that the latter is NP-hard, for any $3 \le k \leq
\lceil n/2\rceil$. We therefore obtain the following result.

\begin{proposition} \label{prop1}
  Deciding whether a given graph is a $k$-word-representable graph, for any given $3\leq k\leq\lceil n/2 \rceil$, is
  NP-complete.
\end{proposition}

It was shown that it is also NP-hard to approximate the dimension of a poset within
$n^{1/2-\epsilon}$-factor \cite{HJ}, and this has recently been strengthened to  and
$n^{1-\epsilon}$-factor \cite{CLN13}.
We therefore obtain the same hardness for the representation number.

\begin{proposition}
  Approximating the representation number within $n^{1-\epsilon}$-factor is NP-hard, for any $\epsilon > 0$. That is, for every constant $\epsilon >
0$, it is an NP-complete problem to decide whether
a given word-representable graph has representation number at most
$n^\epsilon$, or it
has representation number greater than $n^\epsilon$.
\label{prop:approx}
\end{proposition}

\section{Subclasses of Word-Representable Graphs}
\label{classif}

When faced with a new graph class, the most basic questions involve the
kind of properties it satisfies: which known classes are properly
contained (and which not), which graphs are otherwise contained (and
which not), what operations preserve word representability (or
non-representability), and which properties hold for these graphs.

Previously, it was known that the class of word-representable graphs
includes comparability graphs, outerplanar graphs, subdivision
graphs, and prisms. The purpose of this section is to clarify this
situation significantly, including resolving some conjectures. We
start with exploring the relation of colorability and representability.



\begin{corollary}
$3$-colorable graphs are word-representable.
\label{cor:3col}
\end{corollary}

\begin{proof}
Given a 3-coloring of a graph $G$,
direct the edges from the first color class through the second to
the third class. It is easy to see that we obtain a semi-transitive digraph.
Thus, by Theorem.~\ref{thm:rep-equals-semi-trans}, the graph is word-representable.
\end{proof}

This implies a number of earlier results on word-representability,
including that of outerplanar graphs, subdivision graphs, and
prisms \cite{KP}.
The theorem also shows that
2-degenerate graphs (graphs in which every subgraph has a vertex of degree at most 2) and sub-cubic graphs (graphs of maximum degree 3, via Brooks theorem)
are word-representable.

Note, however, that the bound of $2n-4$ for the representability number can be improved for 3-colorable graphs.

\begin{theorem} \label{impr:3col}
$3$-colorable graphs are $2\lfloor 2n/3 \rfloor$-word-representable.
\end{theorem}

\begin{proof}
Let $G$ be a 3-colorable graph.

Suppose that the vertices of $G$ are partitioned into three independent sets $A$, $B$ and
$C$. We may assume that the set $B$ has the largest cardinality among these three sets. Direct the edges of $G$ from $A$ to $B\cup C$ and from $B$ to $C$ and denote the obtained orientation by $D$.


Let $v\in A$. Denote by $N^B_v$ (resp., $N^C_v$) an arbitrary permutation over the set of neighbors of $v$ in $B$ (resp., in $C$) and by $\overline{N^B_v}$  (resp., $\overline{N^C_v}$) --- arbitrary permutations over the remaining vertices of $B$ (resp., of $C$). Also, let $A_v$ be an arbitrary permutation over the set $A \setminus \{v\}$. For a permutation $P$ denote by $R(P)$ the permutation written in the reverse order.

Consider the 2-uniform word
  \[ W^A_v = A_v\ \overline{N^B_v}\ v\ N^B_v\ N^C_v\ v\ \overline{N^C_v}\ R(A_v) \ R(N^B_v) \ R(\overline{N^B_v}) \ R(\overline{N^C_v}) \ R(N^C_v). \]
Note that $W^A_v$ covers all non-edges lying inside $B$ and $C$ and also all non-edges incident with $v$. Indeed, the graph $H$ induced by $W^A_v$ is obtained from the complete tripartite graph with the partition $A,B,C$ by removal of all edges connecting $v$ with $\overline{N^B_v}\cup \overline{N^C_v}$. This justifies 3) and 4). Since $P(W^A_v)=A_v\ \overline{N^B_v}\ v\ N^B_v\ N^C_v\ \overline{N^C_v}$ and no directed path can go from $v$ to $\overline{N^B_v}$, it is a topological sort of $D$, justifying 2).


Similarly, for $v\in C$ consider the 2-uniform word
\[ W^C_v =  N^A_v\ \overline{N^A_v}\ \overline{N^B_v}\  N^B_v\ C_v\ R(\overline{N^A_v})\ v\ R(N^A_v)\ R(N^B_v)\ v\ R(\overline{N^B_v})\ R(C_v), \]
where $C_v$, $\overline{N^A_v}$, $N^A_v$, $\overline{N^B_v}$ and  $N^B_v$ are defined similarly to the respective sets above.
Using the similar arguments, one can show that $W^C_v$ covers all non-edges lying inside $A$ and $B$ and all non-edges incident with $v$.

Concatenating all these words, we obtain a $2k$-uniform word $W$ representing $G$, where $k=|A\cup C|$. Since $B$ has the largest cardinality, $k\le \lfloor 2n/3 \rfloor$.
\end{proof}

Corollary~\ref{cor:3col} does not extend to higher chromatic numbers. The
examples in Fig.~\ref{small} show that 4-colorable graphs can be
non-word-representable. We can, however, obtain a result in terms of the
{\em girth} of the graph, which is the length of its shortest cycle. 

\begin{proposition}
Let $G$ be a graph whose girth is greater than its chromatic number.
Then, $G$ is word-representable.
\end{proposition}

\begin{proof}
Suppose the graph is colored with $\chi(G)$ natural numbers.
Orient the edges of the graph from smaller to larger colors.
There is no directed path with more than $\chi(G)-1$ arcs, but since
$G$ contains no cycle of $\chi(G)$ or fewer edges, there can be no
shortcut. Hence, the digraph is semi-transitive.
\end{proof}

\section{Conclusions}\label{open}

Two of open problems stated in the preliminary version of this paper~\cite{HKP2011} were solved. One of the solved problems is on NP-hardness of the problem of recognition whether a given graph is word-representable or not, and it is discussed in Section~\ref{nice} (see Corollary~\ref{cor:inNP}). The other problem was solved in~\cite{CKL}, where it was shown that there exist non-word-representable graphs of maximum degree 4. We end up the paper with stating revised versions of the remaining two open problems.

\begin{enumerate}
  \setlength{\itemsep}{1pt}
  \setlength{\parskip}{0pt}
  \setlength{\parsep}{0pt}
\item What is the maximum representation number of a graph? We know that it lies between $n/2$
and $2n-4$.
\item Is there an algorithm that forms an $f(k)$-representation of a
$k$-word-representable graph, for some function $f$? Namely, can the
representation number be approximated as a function of itself?
By Prop.~\ref{prop:approx}, this function must grow faster than any fixed polynomial.
\end{enumerate}

\end{document}